\documentclass[12pt]{amsart}

\usepackage{amsthm}

\usepackage{amssymb}
\usepackage{verbatim}
\usepackage[curve,matrix,arrow]{xy}
\usepackage{amsmath,amsfonts}
\usepackage{graphicx}
\usepackage{epsf}

\usepackage{amssymb, graphics}

\newcommand{\mathsym}[1]{{}}

\newcommand{\thmref}[1]{Theorem~\ref{#1}}
\newcommand{\propref}[1]{Proposition~\ref{#1}}
\newcommand{\lemref}[1]{Lemma~\ref{#1}}
\newcommand{\eqnref}[1]{Equation~(\ref{#1})}
\newcommand{\remref}[1]{Remark~\ref{#1}}
\newcommand{\corref}[1]{Corollary~\ref{#1}}
\newcommand{\exref}[1]{Example~\ref{#1}}
\newcommand{\figref}[1]{Figure~\ref{#1}}

  {\end{list}}

\def\vc{\vec{v}}
\def\Zh{{\mathop{\rm Zh}}}

\def\li{L_{i}}
\def\ri{R_{i}}
\def\ddx{{\frac{d}{dx}}}

\def\RR{{\mathbb R}}

\newtheorem{theorem}{Theorem}[section]

\newtheorem{corollary}[theorem]{Corollary}

\newtheorem{lemma}[theorem]{Lemma}
\newtheorem{proposition}[theorem]{Proposition}

\theoremstyle{example}

\newtheorem{remark}[theorem]{Remark}

\theoremstyle{definition}

\theoremstyle{notation}

\newtheorem{example}[theorem]{Example}

\newcommand{\dd}[1]{\delta_{#1}}
\newcommand{\DD}[1]{\Delta_{#1}}

\newcommand{\hj}[3]{\hat{j}_{#1}(#2,#3)}

\newcommand{\ga}{\Gamma}

\newcommand{\tg}{\tau(\Gamma)}
\newcommand{\ta}[1]{\tau(#1)}

\newcommand{\ee}[1]{E(#1)}
\newcommand{\vv}[1]{V(#1)}
\newcommand{\va}{\upsilon}
\newcommand{\vb}{\text{v} \hspace{0.5 mm}}

\newcommand{\pp}{p_{i}}

\newcommand{\qq}{q_{i}}

\newcommand{\mucan}{{\mu_\text{can}}}

\def\can{{\mathop{\rm can}}}

\def\Zh{{\mathop{\rm Zh}}}
\def\CC{{\mathbb C}}

\def\cC{{\mathcal C}}

\def\<{\langle }
\def\>{\rangle }
\newcommand{\secref}[1]{\S\ref{#1}}

\def\tr{\text{tr}}
\def\diag{\text{diag}}
\newcommand{\am}{\mathrm{A}}
\newcommand{\ay}{\mathrm{Y}}
\newcommand{\ax}{\mathrm{X}}
\newcommand{\dm}{\mathrm{D}}
\newcommand{\mm}{\mathrm{M}}

\newcommand{\idm}{\mathrm{I}}
\newcommand{\jm}{\mathrm{J}}
\newcommand{\hm}{\mathrm{H}}
\newcommand{\om}{\mathrm{O}}
\newcommand{\pmm}{\mathrm{M^+}}
\newcommand{\lm}{\mathrm{L}}
\newcommand{\plm}{\mathrm{L^+}}

\newcommand{\lpq}{l_{pq}}

\newcommand{\plp}{l_{pp}^+}
\newcommand{\plq}{l_{qq}^+}
\newcommand{\plpq}{l_{pq}^+}

\def\elg{\ell (\ga)}

\begin{document}

\title[the tau constant and the discrete laplacian matrix]
{The tau constant and the discrete Laplacian matrix of a metrized graph}

\author{Zubeyir Cinkir}
\address{Zubeyir Cinkir\\
Department of Mathematics\\
University of Georgia\\
Athens, Georgia 30602\\
USA}
\email{cinkir@math.uga.edu}

\thanks{I would like to thank Dr. Robert Rumely for his continued support
and the discussions about this paper.}

\begin{abstract}
We express the tau
constant of a metrized graph in terms of the discrete Laplacian matrix and its pseudo inverse.
\end{abstract}

\maketitle

\section{Introduction}\label{sec introduction}

Metrized graphs are finite graphs equipped with a distance function on their edges. For a metrized graph $\ga$, the tau constant $\tg$ is an invariant which plays important roles in both harmonic analysis on metrized graphs and arithmetic of curves.

T. Chinburg and R. Rumely \cite{CR} introduced a canonical measure $\mu_{can}$ of total mass $1$ on a
metrized graph $\ga$. The diagonal values of the Arakelov-Green's function $g_{\mu_{can}}(x,x)$ associated to $\mu_{can}$ are constant on $\ga$. M. Baker and Rumely called this constant ``the tau constant'' of a metrized graph $\ga$, and denoted it by $\tg$. They \cite[Conjecture 14.5]{BRh} posed a conjecture concerning the existence of a universal lower bound for $\tg$. We call it Baker and Rumely's lower bound conjecture.

Baker and Rumely \cite{BRh} introduced a measure valued Laplacian
operator $\Delta$ which extends Laplacian operators studied earlier
in \cite{CR} and \cite{Zh1}. This Laplacian operator combines the
``discrete'' Laplacian on a finite graph and the ``continuous''
Laplacian $-f''(x)dx$ on $\RR$. In terms of spectral theory, the tau constant $\tg$ is the trace of the inverse operator of $\Delta$ with respect to $\mu_{can}$ when $\ga$ has total length $1$.

The results in \cite{Zh2}, \cite[Chapter 4]{C1} and \cite{C4} indicate that the tau constant has important applications
in arithmetic of curves such as its connection to the Effective Bogomolov Conjecture over function fields.

In the article \cite{C2}, various formulas for $\tg$ are given, and Baker and Rumely's lower bound conjecture is verified for a number of large families of graphs. It is shown in the article \cite{C3} that this conjecture holds for metrized graphs with edge connectivity more than $4$; and proving it for cubic graphs is sufficient to show that it holds for all graphs.

Verifying the Baker and Rumely's lower bound conjecture in the remaining cases or showing a counter example to this conjecture, and finding metrized graphs with minimal tau constants are interesting and subtle problems.
However, except for some special cases, computing the tau constant for metrized graphs with large number of vertices is not an easy task. In this paper, we will give a formula for the tau constant of $\ga$ in terms of the discrete Laplacian matrix $\lm$ of $\ga$ and its pseudo inverse $\plm$. In particular, this formula leads to rapid computation of $\tg$ by using computer softwares.

In \secref{sec tau constant}, we briefly introduce metrized graphs, Laplacian
operator $\Delta$, the canonical measure $\mu_{can}$ and the tau constant $\tg$. We revise the fact that metrized graphs can be interpreted as electric circuits. At the end of \secref{sec tau constant}, we give several formulas
concerning the tau constant. In \secref{sec discrete laplacian}, we introduce the discrete Laplacian matrix $\lm$
of a metrized graph. We recall some of the properties of $\lm$ and $\plm$.
We start \secref{sec taudisc} with a remarkable relation between the resistance on $\ga$ and the pseudo inverse of the discrete Laplacian on $\ga$ \cite{RB2}. Then we derive a number of new identities by combining this relation with the results from \secref{sec tau constant} and \secref{sec discrete laplacian}. Finally, we express the canonical measure in terms of $\lm$ and $\plm$, and obtain our main result which is the following theorem:
\begin{theorem}\label{thm disc2 copy}
Let $\lm=(l_{p\,q})_{v \times v}$ be the discrete Laplacian matrix of a metrized graph
$\ga$, and let $\plm=(l_{p\,q}^+)_{v \times v}$ be its pseudo inverse.
Suppose $\pp$ and $\qq$ are the end points of edge $e_i$ of $\ga$ for each $i
=1, \, 2,\cdots,e$, where $e$ is the number of edges in $\ga$. Then we have
\begin{equation*}
\begin{split}
\tg  =-\frac{1}{12}\sum_{e_i \in \ee{\ga}} l_{\pp \qq}\big(\frac{1}{ l_{\pp \qq}}+l_{\pp \pp}^+-2l_{\pp
\qq}^+ +l_{\qq \qq}^+ \big)^2 +\frac{1}{4}\sum_{q, \, s  \in
\vv{\ga}} l_{qs} l_{qq}^+  l_{ss}^+ +\frac{1}{v}trace(\plm).
\end{split}
\end{equation*}
\end{theorem}
We prove \thmref{thm disc2 copy} at the end of \secref{sec taudisc}; and we give two examples for the computations of $\tg$ and $\mu_{can}$.

Note that there is a $1-1$ correspondence between the equivalence classes of finite connected weighted
graphs, the metrized graphs, and the resistive electric circuits.
If an edge $e_i$ of a metrized graph has length $\li$, then we have that the resistance along $e_i$ is $\li$
in the corresponding resistive electric circuit, and that the weight of $e_i$ is $\frac{1}{\li}$ in the corresponding
weighted graph. The identities we show for metrized graphs in this paper are also valid
for electrical networks, and they have equivalent forms on a weighted graph.

The results in this paper are more clarified and organized versions of those given in \cite[Sections 5.1, 5.2, 5.3 and 5.4]{C1}.

\section{The tau constant of a metrized graph}\label{sec tau constant}

A metrized graph $\ga$ is a finite connected graph such that its edges are equipped with a distinguished
parametrization. One can find other definitions of metrized graphs in \cite{RumelyBook}, \cite{CR}, \cite{BRh}, \cite{Zh1}, and \cite{BF}.

A metrized graph can have multiple edges and self-loops. For any given $p \in \ga$,
the number of directions emanating from $p$ will be called the \textit{valence} of $p$, and will be denoted by
$\va(p)$. By definition, there can be only finitely many $p \in \ga$ with $\va(p)\not=2$.

For a metrized graph $\ga$, we will denote its set of vertices by $\vv{\ga}$.
We require that $\vv{\ga}$ be finite and non-empty and that $p \in \vv{\ga}$ for each $p \in \ga$ if
$\va(p)\not=2$. For a given metrized graph $\ga$, it is possible to enlarge the
vertex set $\vv{\ga}$ by considering more additional points of valence $2$ as vertices.

For a given graph $\ga$ with vertex set $\vv{\ga}$, the set of edges of $\ga$ is the set of closed line segments with end points in $\vv{\ga}$. We will denote the set of edges of $\ga$ by $\ee{\ga}$. However, we will denote the graph obtained from $\ga$ by deletion of the interior points of an edge $e_i \in \ee{\ga}$ by $\ga-e_i$.

We denote $\# (\vv{\ga})$ and $\# (\ee{\ga})$ by $v$ and $e$, respectively.
We denote the length of an edge $e_i \in \ee{\ga}$ by $\li$. The total length of $\ga$, which will be denoted by $\elg$, is given by $\elg=\sum_{i=1}^e\li$.
%

Let $\Zh(\Gamma)$ be the set of all
continuous functions $f : \Gamma \rightarrow \CC$ such that for some vertex set $\vv{\ga}$, $f$ is
$\cC^2$ on $\ga
\backslash \vv{\ga}$ and $f^{\prime \prime}(x) \in L^1(\Gamma)$.
Baker and Rumely \cite{BRh} defined the following measure valued \textit{Laplacian} on a given metrized graph.
For a function $f \in \Zh(\Gamma)$,
\begin{equation}
\DD{x}(f(x))=-f''(x)dx - \sum_{p \in \vv{\ga}}\bigg[ \sum_{\vc
\hspace{0.5 mm} \text{at} \hspace{0.5 mm} p}
d_{\vc}f(p)\bigg]\dd{p}(x),
\end{equation}
See the article \cite{BRh} for details and for a description of the largest class of functions for which a measure valued Laplacian can be defined.

In the article \cite{CR}, a kernel $j_{z}(x,y)$ giving a
fundamental solution of the Laplacian is defined and studied as a
function of $x, y, z \in \Gamma$. For fixed $z$ and $y$ it has the
following physical interpretation: When $\Gamma$ is viewed as a
resistive electric circuit with terminals at $z$ and $y$, with the
resistance in each edge given by its length, then $j_{z}(x,y)$ is
the voltage difference between $x$ and $z$, when unit current enters
at $y$ and exits at $z$ (with reference voltage 0 at $z$).

For any $x$, $y$, $z$ in $\ga$, the voltage function $j_z(x,y)$ on
$\ga$ is a symmetric function in $x$ and $y$, and it satisfies
$j_x(x,y)=0$ and $j_x(y,y)=r(x,y)$, where $r(x,y)$ is the resistance
function on $\ga$. For each vertex set $\vv{\ga}$, $j_{z}(x,y)$ is
continuous on $\ga$ as a function of $3$ variables.
As the physical interpretation suggests, $j_z(x,y) \geq 0$ for all $x$, $y$, $z$ in $\ga$.
For proofs of these facts, see the articles \cite{CR}, \cite[sec 1.5 and sec 6]{BRh}, and \cite[Appendix]{Zh1}.
The voltage function $j_{z}(x,y)$ and the resistance function $r(x,y)$ on a metrized graph
were also studied in the articles \cite{BF}, \cite{C2}.

For any real-valued, signed Borel measure $\mu$ on $\Gamma$ with
$\mu(\Gamma)=1$ and $|\mu|(\Gamma) < \infty$, define the function
$j_{\mu}(x,y) \ = \ \int_{\Gamma} j_{\zeta}(x,y) \, d\mu({\zeta}).$
Clearly $j_{\mu}(x,y)$ is symmetric, and is jointly continuous in
$x$ and $y$. T. Chinburg and Rumely \cite{CR} discovered that there is a unique real-valued, signed Borel measure $\mu=\mu_{can}$ such that $j_{\mu}(x,x)$ is constant on $\ga$. The measure $\mu_\can$ is called the
\textit{canonical measure}.
Baker and Rumely \cite{BRh} called the constant $\frac{1}{2}j_{\mu}(x,x)$ the \textit{tau constant} of $\ga$ and denoted it by $\tg$.
%

\begin{lemma}\cite[Corollary 14.3]{BRh}\label{cor:tauisatrace}
Let $ \{ \lambda_1, \, \lambda_2, \, \lambda_3, \, \ldots \}$ be the set of eigenvalues of the Laplacian $\Delta$ with respect to the canonical measure $\mu_\can$. Then
$$
\ell(\Gamma) \cdot \tau(\Gamma) = \sum_{n=1}^\infty \frac{1}{\lambda_n},
$$

In particular, If $\ell(\Gamma)=1$, then $\tau(\Gamma)$ is the trace of the inverse operator of $\Delta$
with respect to $\mu_{can}$.
\end{lemma}

%
%
%

The following theorem gives an explicit description of the canonical measure $\mu_{can}$:
\begin{theorem}\cite[Theorem 2.11]{CR} \label{thmCanonicalMeasureFormula}
Let $\ga$ be a metrized graph. Suppose that $\li$ is the length of edge $e_i$ and $R_i$ is the effective
resistance between the endpoints of $e_i$ in the graph $\Gamma-e_i$.
Then we have
\begin{equation*}
\mu_\can(x) \ = \ \sum_{p \in \vv{\ga}} (1 - \frac{1}{2}\vb(p))
\, \delta_p(x) + \sum_{e_i \in \ee{\ga}} \frac{dx}{L_i+R_i},
\end{equation*}
where
$\delta_p(x)$ is the Dirac measure.
\end{theorem}
Here is another expression for $\tg$:
\begin{lemma}\cite{REU}\label{lemtau1}
For any metrized graph $\ga$ and its resistance function $r(x,y)$,
$$\tau(\Gamma) = \frac{1}{2}\int_{\ga} r(x,y) d\mucan(y).$$
\end{lemma}
Another description of $\tau(\Gamma)$ is as follows:
\begin{lemma}\cite[Lemma 14.4]{BRh}\label{lemtauformula}
For any fixed $p \in \Gamma$, we have
$
\tau(\Gamma) = \frac{1}{4} \int_\Gamma
\left( \ddx r(x,p) \right)^2 dx .
$
\end{lemma}
\begin{remark}
\label{remvalence}
Let $\ga$ be any metrized graph with resistance function $r(x,y)$.
If we enlarge $\vv{\ga}$ by including points $p \in
\ga$ with $\va(p)=2$, the resistance function does not change, and thus $\tg$ does not change by \lemref{lemtauformula}.
\end{remark}
Note that $\tg$ is an invariant of the metrized graph $\ga$, which depends only
on the topology and the edge length distribution of $\ga$.

Let $\ga-e_i$ be a connected graph for an edge $e_i \in \ee{\ga}$ of length $\li$. Suppose $\pp$ and $\qq$ are the end points of $e_i$, and $p \in \ga-e_i$. By applying circuit reductions,
$\ga-e_i$ can be transformed into a $Y$-shaped graph with the same resistances between $\pp$, $\qq$, and $p$ as in $\ga-e_i$.
More details on this can be found in \cite[Section 2]{C2}.
Since $\ga-e_i$ has such circuit reduction, $\ga$ has the circuit reduction as illustrated in \figref{fig 2termpnew3}  with  the corresponding voltage values on each segment, where $\hj{x}{y}{z}$ is the voltage function in $\ga-e_i$. Throughout this paper, we will use the following notation:
$R_{a_i,p} := \hj{\pp}{p}{\qq}$, $R_{b_i,p} := \hj{\qq}{\pp}{p}$, $R_{c_i,p} := \hj{p}{\pp}{\qq}$, and $\ri$ is the resistance
between $\pp$ and $\qq$ in $\ga-e_i$. Note that $R_{a_i,p}+R_{b_i,p}=\ri$ for each $p \in \ga$. When $\ga-e_i$ is not connected, we set $R_{b_i,p}=\ri=\infty$ and $R_{a_i,p}=0$ if $p$ belongs to the component of $\ga-e_i$
containing $\pp$, and we set $R_{a_i,p}=\ri=\infty$ and $R_{b_i,p}=0$ if $p$ belongs to the component of $\ga-e_i$
containing $\qq$.
\begin{figure}
\centering
\includegraphics[scale=0.7]{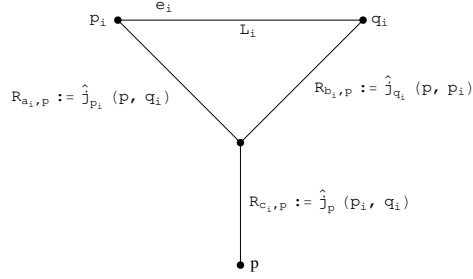} \caption{Circuit reduction of $\ga-e_i$ with reference to $\pp$, $\qq$ and $p$.} \label{fig 2termpnew3}
\end{figure}

By computing the integration in \lemref{lemtauformula}, one obtains the following formula for the tau constant:
\begin{proposition}\cite{REU} \label{proptau}
Let $\Gamma$ be a metrized graph, and let $L_i$ be the length of the
edge $e_{i}$, for $i \in \{1,2, \dots, e\}$.
Using the notation above,
if we fix a vertex $p$ we have
\[
\ta{\ga} = \frac{1}{12} \sum_{e_i \in \ga} \left(
\frac{\li^3+3\li(R_{a_{i},p}-R_{b_{i},p})^2}{(\li+\ri)^2}\right ).
\]
Here, if $\ga-e_i$ is not connected, i.e. $\ri$ is infinite, the
summand corresponding to $e_i$ should be replaced by $3\li$, its limit as $\ri \longrightarrow
\infty$.
\end{proposition}
The proof of \propref{proptau} can be found in \cite[Proposition 2.9]{C2}.
We will use the following remark in \secref{sec taudisc}.
\begin{remark}\label{rem independence pdif}
It follows from \lemref{lemtauformula} and \propref{proptau} that
$\sum_{e_i \in \ee{\ga}}
\frac{\li(R_{a_{i},p}-R_{b_{i},p})^2}{(\li+\ri)^2}$ is independent
of the chosen vertex $p \in \vv{\ga}$.
\end{remark}
Let $\pp$ and $\qq$ be the end points of the edge $e_i$ as in \figref{fig 2termpnew3}. It follows from parallel and series reductions that
\begin{equation}\label{eqn2term0}
\begin{split}
r(p_i,p)=\frac{(\li+R_{b_i,p})R_{a_i,p}}{\li+\ri}+R_{c_i,p},
\quad \text{and} \quad r(q_i,p)=\frac{(\li+R_{a_i,p})R_{b_i,p}}{\li+\ri}+R_{c_i,p}.
\end{split}
\end{equation}
Therefore, $r(\pp,p)-r(\qq,p) = \frac{\li(R_{a_i,p}-R_{b_i,p})}{\li+\ri}$,
and so
\begin{equation}\label{eqn2termeq1}
\begin{split}
\sum_{e_i \in \,
\ee{\ga}}\frac{\li(R_{a_{i},p}-R_{b_{i},p})^2}{(\li+\ri)^2}=
\sum_{e_i \in \, \ee{\ga}}\frac{(r(\pp,p)-r(\qq,p))^2}{\li}.
\end{split}
\end{equation}
\begin{proposition}\label{prop tau canonical}
Let $\ga$ be a metrized graph with the resistance function $r(x,y)$, and let each edge $e_i \in \ee{\ga}$
be parametrized by a segment $[0,\li]$, under its arclength parametrization. Then for any $p \in \vv{\ga}$,
$$ \tg = -\frac{1}{4} \sum_{q \in  \vv{\ga}}(\va(q)-2)r(p,q)
+ \frac{1}{2} \sum_{e_i \in \, \ee{\ga}} \frac{1}{\li+\ri} \int_{0}^{\li} r(p,x) dx.$$
\end{proposition}
\begin{proof} We have $\tg = \frac{1}{2}\int_{\ga} r(p,x) d\mucan(x)$, by
\lemref{lemtau1}. Then by \thmref{thmCanonicalMeasureFormula},
\begin{equation*}
\begin{split}
\tg = \frac{1}{2}\sum_{q \in \vv{\ga} } (1 - \frac{1}{2}\va(p)) \int_{\ga} r(p,x) \delta_q(x)
+ \sum_{e_i \in \, \ee{\ga}} \frac{1}{\li+\ri} \int_{0}^{\li} r(p,x) dx .
\end{split}
\end{equation*}
This gives the result.
\end{proof}
\begin{lemma}\label{lem crtterm}
Let $\pp$ and $\qq$ be end points of $e_i \in \ee{\ga}$. For any $p
\in \vv{\ga}$,
\begin{equation*}
\begin{split}
\sum_{e_i \in \,
\ee{\ga}}\frac{\li(R_{a_{i},p}-R_{b_{i},p})^2}{(\li+\ri)^2} &
=\sum_{e_i \in \, \ee{\ga}}\frac{\li}{\li + \ri}
\big(r(\pp,p)+r(\qq,p)\big) - \sum_{q \in
\vv{\ga}}(\va(q)-2)r(p,q).
\end{split}
\end{equation*}
\end{lemma}
\begin{proof}
We first note that $r(x,p)=\frac{(x+R_{a_{i},p})(\li-x+R_{b_{i},p})}{\li+\ri}+R_{c_{i},p}$ if
$x \in e_i$.
By \lemref{lemtauformula}, $ 4 \tau(\Gamma) = \int_\Gamma
\left( \ddx r(x,y) \right)^2 dx.$ Thus, integration by parts gives
\begin{equation}\label{eqn crtterm1}
\begin{split}
4 \tg &=\sum_{e_i \in \, \ee{\ga}} \big(r(p,x) \cdot \ddx
r(p,x)\big)|^{\li}_{0} - \sum_{e_i \in \, \ee{\ga}} \int_{0}^{\li}
r(p,x) \frac{d^2}{dx^2}r(p,x) dx.
\end{split}
\end{equation}
Since $\frac{d^2}{dx^2}r(p,x)=\frac{-2}{\li+\ri}$ if $x \in e_i$,
the result follows from \propref{prop tau canonical} and Equations
(\ref{eqn2term0}) and (\ref{eqn crtterm1}).
\end{proof}
Chinburg and Rumely \cite[page 26]{CR} showed that
\begin{equation}\label{eqn genus}
\sum_{e_i \in \ee{\ga}}\frac{\li}{\li +\ri}=e-v+1.
\end{equation}

\section{The discrete Laplacian matrix $\lm$ and its pseudo inverse $\plm$.}\label{sec discrete laplacian}
\vskip .1 in

Throughout this paper, all matrices will have entries in $\RR$. To
have a well-defined discrete Laplacian matrix $\lm$ for a metrized
graph $\ga$, we first choose a vertex set $\vv{\ga}$ for $\ga$ in
such a way that there are no self-loops, and no multiple edges
connecting any two vertices. This can be done for any graph $\ga$ by
enlarging the vertex set by considering additional valence two points as vertices
whenever needed. We will call such a vertex set $\vv{\ga}$ \textit{optimal}.
If distinct vertices $p$ and $q$ are the end points of an edge, we
call them \textit{adjacent} vertices.

Given a matrix $\mm$, let $\mm^{T}$, $tr(\mm)$, $\mm^{-1}$
be the transpose, trace and inverse of  $\mm$, respectively.
Let $\idm_v$ be the $v \times v$ identity matrix, and let $\om$ be
the zero matrix (with the appropriate size if it is not specified).
Let $\jm$ be an $v \times v$ matrix having each entries $1$.

Let $\ga$ be a metrized graph with $e$ edges and with an optimal vertex set
$\vv{\ga}$ containing $v$ vertices. Fix an ordering of the vertices
in $\vv{\ga}$. Let $\{L_1, L_2, \cdots, L_e\}$ be a labeling of the
edge lengths. The matrix $\am=(a_{pq})_{v \times v}$ given by
\[
a_{pq}=\begin{cases} 0 & \quad \text{if $p = q$, or $p$ and $q$ are
not adjacent}.\\
\frac{1}{L_k} & \quad \text{if $p \not= q$, and $p$ and $q$ are
connected
by} \text{ an edge of length $L_k$}\\
\end{cases}
\]
is called the \textit{adjacency matrix} of $\ga$. Let $\dm=\diag(d_{pp})$ be
the $v \times v$ diagonal matrix given by $d_{pp}=\sum_{s \in
\vv{\ga}}a_{ps}$. Then $\lm:=\dm-\am$ is called the \textit{discrete
Laplacian matrix} of $\ga$. That is, $ \lm =(l_{pq})_{v \times v}$ where
\[
l_{pq}=\begin{cases} 0 & \; \, \text{if $p \not= q$, and $p$ and $q$
are not adjacent}.\\
-\frac{1}{L_k} & \; \, \text{if $p \not= q$, and $p$ and $q$ are
connected by} \text{ an edge of length $L_k$}\\
-\sum_{s \in \vv{\ga}-\{p\}}l_{ps} & \; \, \text{if $p=q$}
\end{cases}.
\]
The discrete Laplacian matrix is also known as the generalized (or
the weighted) Laplacian matrix in the literature.
A matrix $\mm$ is called \textit{doubly centered}, if both row and column
sums are $0$. That is, $\mm$ is doubly centered iff $\mm \ay= \om$
and $\ay^T \mm =\om$, where $\ay=[1,1,\cdots,1]^T$.
\begin{example}\cite[Remark 3.1]{C4}\label{ex disc0}
For any metrized graph $\ga$, the discrete Laplacian matrix $\lm$ is symmetric and doubly
centered. That is, $\sum_{p \in \vv{\ga}} \lpq = 0$ for each $q \in
\vv{\ga}$, and $\lpq =l_{qp}$ for each $p$, $q$ $\in \vv{\ga}$.
\end{example}
In our case, $\ga$ is connected by definition. Thus, the discrete Laplacian matrix $\lm$ of $\ga$ is a $(v
\times v)$ matrix of rank $v-1$ if the optimal vertex set $\vv{\ga}$ has $v$ vertices. The null space of $\lm$ is the
$1$-dimensional space spanned by $[1,1,\cdots,1]^T$. Since $\lm$ is
a real symmetric matrix, it has real eigenvalues. Moreover, $\lm$ is
positive semi-definite
More precisely, one of the eigenvalues of $\lm$ is $0$ and the others are positive. Thus, $\lm$
is not invertible. However, it has generalized inverses. In
particular, it has the pseudo inverse $\plm$, also known as the
Moore-Penrose generalized inverse, which is uniquely determined by
the following properties:
\begin{align*}
i)  \quad &\lm \plm \lm  = \lm, & \qquad \qquad
iii) \quad &(\lm \plm)^{T}   = \lm \plm,
\\ ii) \quad &\plm \lm \plm  = \plm, & \qquad \qquad
iv) \quad &(\plm \lm)^{T}   = \plm \lm .
\end{align*}

An $v \times v$ matrix $\mm$ is called an \textit{EP-matrix} if $\pmm \mm=
\mm \pmm$. A necessary and sufficient condition for $\mm$ to be an
EP-matrix is that $\mm u = \lambda u $ iff $\pmm u = \lambda^+ u$,
for each eigenvector $u$ of $\mm$. Another characterization of an
EP-matrix $\mm$ is that $\mm \ax = \om$ iff $\mm^T \ax= \om$, where
$\ax$ is also $v \times v$. Any symmetric matrix is an EP-matrix (\cite[pg 253]{SB}).


We have the following properties:
\begin{align*}
i)  \quad &\lm \text{ and } \plm \text{ are symmetric}, & \qquad \qquad
iii) \quad &\lm \text{ and } \plm \text{ are EP matrices},
\\ ii) \quad &\lm \text{ and } \plm \text{ are doubly centered}, & \qquad \qquad
iv) \quad &\lm \text{ and } \plm \text{ are positive
semi-definite}.
\end{align*}

For a discrete Laplacian matrix $\lm$ of size $v \times v$, we have the following
formula for $\plm$ (see \cite[ch 10]{C-S}):
\begin{equation}\label{eqn disc5}
 \plm = \big( \lm - \frac{1}{v}\jm \big)^{-1} + \frac{1}{v} \jm.
\end{equation}
where $\jm$ is of size $v \times v$ and has all entries $1$.
\begin{remark}\label{rem doubcent1}
Since $\plm$ is doubly centered, $\sum_{p \in \vv{\ga}} \plpq =0$,
for each $q \in \vv{\ga}$. Also, $\plpq = l_{qp}^+$, for each $p$,
$q$ $\in \vv{\ga}$.
\end{remark}

We use the following lemma and its corollary, \corref{cor disc1}, frequently in the rest of this article:
\begin{lemma}\cite[Equation 2.9]{D-M}\label{lem disc1}
Let $\jm$ be of size $v \times v$ as above and let $\lm$ be the
discrete Laplacian of a graph (not necessarily with equal edge lengths).
Then $\lm \plm = \plm \lm = \idm -\frac{1}{v}\jm$.
\end{lemma}
\begin{corollary}\label{cor disc1}
Let $\ga$ be a metrized graph and let $\lm$ be the corresponding discrete
Laplacian matrix of size $v \times v$. Then for any $p$, $q$ $\in
\vv{\ga}$,
$$\sum_{s \in \vv{\ga}}
l_{ps}^+ l_{sq}=\begin{cases} -\frac{1}{v} & \text{ if $p \not =q$}\\
\frac{v-1}{v} & \text{ if $p=q$ } \end{cases}.$$
\end{corollary}

See \cite{C-S}, \cite[ch 10]{SB}, \cite{RB1} and \cite{RM} for more
information about $\lm$ and $\plm$.

\section{The discrete Laplacian, the resistance function, and the tau constant}
\label{sec taudisc} \vskip .1 in


In this section, we will obtain a formula (see \thmref{thm disc2})
for the tau constant in terms of the entries of $\lm$ and $\plm$.
Our main tools will be a remarkable relation between the resistance
and the pseudo inverse $\plm$ (\lemref{lem disc2} below), properties
of $\lm$ and $\plm$ given in \secref{sec discrete laplacian}, the
results from \secref{sec introduction} concerning metrized graphs, and the circuit reduction
theory.
\begin{lemma} \cite{RB2}, \cite{RB3}, \cite[Theorem A]{D-M} \label{lem disc2}
Suppose $\ga$ is a graph with the discrete Laplacian $\lm$ and the
resistance function $r(x,y)$. Let $\hm$ be a generalized inverse
of $\lm$ (i.e., $\lm \hm \lm = \lm$). Then we have
$$r(p,q)=\hm_{pp}-\hm_{pq}-\hm_{qp}+\hm_{qq}, \quad \text{for any $p$, $q$ $\in \vv{\ga}$}.$$
In particular, for the pseudo inverse $\plm$ we have
$$r(p,q)=l_{pp}^+-2l_{pq}^+ + l_{qq}^+, \quad \text{for any $p$, $q$ $\in \vv{\ga}$}.$$
\end{lemma}
\lemref{lem disc2} shows that the pseudo inverses can be used to compute the resistance $r(p,q)$ between any $p$, $q$ in $\ga$. Namely, we choose an optimal vertex set $\vv{\ga}$
containing $p$ and $q$. Then we compute the corresponding pseudo inverse, and apply \lemref{lem disc2}. Similarly, the following lemma shows that the pseudo inverses can be used to compute the voltage $j_p(q,s)$ for any $p$, $q$ and $s$ in $\ga$.
\begin{lemma}\cite[Lemma 3.5]{C4}\label{lem voltage1}
Let $\ga$ be a graph with the discrete Laplacian $\lm$ and the voltage
function $j_x(y,z)$. Then for any $p$, $q$, $s$ in $\vv{\ga}$,
$$j_p(q,s)=l_{pp}^+ - l_{pq}^+ -l_{ps}^+ +l_{qs}^+.$$
\end{lemma}
\begin{corollary}\label{cor lpq comp}
Let $\ga$ be a graph with the discrete Laplacian matrix $\lm$ having the pseudo inverse $\plm$.
Then for any $p, \, q \in \vv{\ga}$, we have
$ l_{pp}^+ \geq l_{pq}^+ .$
\end{corollary}
\begin{proof} By \remref{rem doubcent1} and \lemref{lem voltage1},
$\sum_{s \in \vv{\ga}} j_p(q,s) = v \cdot (l_{pp}^+ - l_{pq}^+)$ for any $p$ and $q$ in $\vv{\ga}$.
Thus the result follows from the fact that $j_p(q,s) \geq 0$ for any $p,
\, q, \, s \in \ga$.
\end{proof}

Recall that we use $\li$ for the length of edge $e_i \in \ee{\ga}$ and $R_i$ for the resistance between the endpoints of $e_i$ in the graph $\Gamma-e_i$.
An important term for computations concerning $\tg$ is expressed in terms of $\lm$ and
$\plm$ by the following lemma:
\begin{lemma}\label{lem disc8}
Let $\lm$ be the discrete Laplacian matrix of size $v \times v$ for a
graph $\ga$. Let $\pp$ and $\qq$ be the end points of edge $e_i$  for
any given $e_i \in \ee{\ga}$. Then
$$\sum_{e_i \in \ee{\ga}} \frac{\li \ri^2}{(\li + \ri)^2} =
\frac{4(v-1)}{v} \tr(\plm)- \sum_{p, \, q \in \vv{\ga}} \lpq \plp
\plq -2 \sum_{p, \, q \in \vv{\ga} } \lpq \big(l_{pq}^+ \big)^2.$$
\end{lemma}
\begin{proof}
First, we use \exref{ex disc0} to obtain
\begin{equation}\label{eqnlem disc6}
\begin{split}
\sum_{p, \, q \in \vv{\ga}} \lpq \big(\plp \big)^2 =\sum_{p \in
\vv{\ga}} \big(\plp \big)^2 \Big( \sum_{q \in \vv{\ga}} \lpq \Big)=0.
\end{split}
\end{equation}
Using \corref{cor disc1},
\begin{equation}\label{eqnlem disc5}
\begin{split}
\sum_{p, \, q \in \vv{\ga}} \lpq \plpq \plp = \frac{v-1}{v} \cdot \tr(\plm).
\end{split}
\end{equation}
Then
\begin{equation*}
\begin{split}
\sum_{e_i \in \ee{\ga}} \frac{\li \ri^2}{(\li + \ri)^2}
& = \sum_{e_i \in \ee{\ga}} \frac{1}{\li}\big(r(\pp,\qq) \big)^2, \quad \text{ since $r(\pp,\qq)=\frac{\li \ri}{\li+\ri}$.} \\
& = -\sum_{e_i \in \ee{\ga}} l_{\pp \qq}\big( l_{\pp \pp}^+ +l_{\qq
\qq}^+ -2l_{\pp,\qq}^+  \big)^2,
\quad \text{by \lemref{lem disc2}.} \\
& = -\frac{1}{2}\sum_{p, \, q \in \vv{\ga}} \lpq \big( l_{pp}^+
+l_{qq}^+ -2l_{pq}^+ \big)^2, \quad \text{as $l_{pq}=0$ if $p$, $q$ are not adjacent.} \\
& = -\frac{1}{2}\sum_{p, \, q \in \vv{\ga}} \lpq \big( l_{pp}^+
+l_{qq}^+ \big)^2 + 2\sum_{p, \, q \in \vv{\ga}} \Big( \lpq (l_{pp}^+
+l_{qq}^+)l_{pq}^+ -\lpq \big( l_{pq}^+ \big)^2 \Big)
\\
& = -\sum_{p, \, q \in \vv{\ga}} \lpq l_{pp}^+ l_{qq}^+ +
\sum_{p, \, q \in \vv{\ga}} \Big( 4 \lpq l_{pp}^+ l_{pq}^+ -2 \lpq \big( l_{pq}^+ \big)^2 \Big)
, \quad \text{by \eqnref{eqnlem disc6}.} \\
\end{split}
\end{equation*}
Thus, the result follows from \eqnref{eqnlem disc5}.
\end{proof}

Next, we will have several lemmas concerning identities involving the entries of $\lm$ and $\plm$.
\begin{lemma}\label{lem disc9}
Let $\lm$ be the discrete Laplacian matrix of a
graph $\ga$. Then for any $p \in \vv{\ga}$,
$$\sum_{q, \, s \in \vv{\ga}} l_{qs} \big(l_{qq}^+ -l_{ss}^+ \big) \big(l_{qp}^+ -l_{sp}^+ \big) =
-2\sum_{q, \, s \in \vv{\ga}} l_{qs} l_{qq}^+ l_{sp}^+.$$
\end{lemma}
\begin{proof}
By using \exref{ex disc0}, for any $p \in \vv{\ga}$
\begin{equation}\label{eqnlem disc9a}
\begin{split}
& \sum_{q, \, s \in \vv{\ga}} l_{qs} l_{qq}^+ l_{qp}^+ = \sum_{q \in
\vv{\ga}} l_{qq}^+ l_{qp}^+ \Big( \sum_{ s \in \vv{\ga}} l_{qs}
\Big) = 0.
\end{split}
\end{equation}
Using \exref{ex disc0} and \eqnref{eqnlem disc9a} for the second equality,
\begin{equation*}
\begin{split}
\sum_{q, \, s \in \vv{\ga}} l_{qs} \big(l_{qq}^+ -l_{ss}^+ \big)
\big(l_{qp}^+ -l_{sp}^+ \big)
&= \sum_{q, \, s \in \vv{\ga}} \Big( l_{qs} l_{qq}^+ l_{qp}^+ -l_{qs} l_{qq}^+
l_{sp}^+ -l_{qs} l_{ss}^+ l_{qp}^+ + l_{qs} l_{ss}^+ l_{sp}^+\Big)
\\ & = -\sum_{q, \, s \in \vv{\ga}} \big( l_{qs} l_{qq}^+ l_{sp}^+ + l_{qs} l_{ss}^+ l_{qp}^+ \big).
\end{split}
\end{equation*}
This is equivalent to what we wanted.
\end{proof}
\begin{lemma}\label{lem disc9r}
Let $\lm$ be the discrete Laplacian matrix of size $v \times v$ for a
graph $\ga$, and let $\pp$, $\qq$ be the end points of $e_i \in
\ee{\ga}$. Then for any $p \in \vv{\ga}$,
\begin{equation*}
\begin{split}
l_{pp}^+ & = \frac{1}{v}tr(\plm)+ \sum_{e_i \in \ee{\ga}}
\frac{\li}{\li+\ri} \big( l_{p \pp}^+ + l_{p \qq}^+ \big) - \sum_{q
\in \vv{\ga}} \va(q) l_{pq}^+, \\
l_{pp}^+ & = \frac{1}{v}tr(\plm)- \sum_{e_i \in \ee{\ga}} \frac{\ri}{\li+\ri}
\big( l_{p \pp}^+ + l_{p \qq}^+ \big).
\end{split}
\end{equation*}
\end{lemma}
\begin{proof}
We use \lemref{lem crtterm} for the first equality below and \lemref{lem disc2} for the second equality below:
\begin{equation*}\label{eqn discr}
\begin{split}
& \sum_{e_i \in \ee{\ga}}\frac{\li(R_{a_{i},p}-R_{b_{i},p})^2}{(\li+\ri)^2} = \sum_{e_i
\in \ee{\ga}}\frac{\li}{\li + \ri} \big(r(\pp,p)+r(\qq,p)\big) -
\sum_{q \in \vv{\ga}}(\va(q)-2)r(p,q)
\\
&= \sum_{e_i \in \ee{\ga}}\frac{\li}{\li + \ri} \big( l_{\pp
\pp}^+ + l_{\qq \qq}^+ -2(l_{p \pp}^+ +l_{p \qq}^+ -l_{pp}^+) \big)
- \sum_{q \in \vv{\ga}}(\va(q)-2) \big(l_{qq}^+ -2l_{pq}^+ +l_{pp}^+ \big).
\\
&= \sum_{e_i \in \ee{\ga}}\frac{\li}{\li + \ri} \big( l_{\pp
\pp}^+ + l_{\qq \qq}^+ -2(l_{p \pp}^+ +l_{p \qq}^+ ) \big) +
2l_{pp}^+ - \sum_{q \in \vv{\ga}}(\va(q)-2)l_{qq}^+ +2 \sum_{q \in
\vv{\ga}}\va(q)l_{pq}^+.
\\ & \text{by \eqnref{eqn genus} and the fact that $\sum_{q \in \vv{\ga}}(\va(q)-2)=2e-2v$.}
\end{split}
\end{equation*}
Since $\sum_{e_i \in \ee{\ga}}\frac{\li(R_{a_{i},p}-R_{b_{i},p})^2}{(\li+\ri)^2}=\frac{1}{v}\sum_{p
\in \vv{\ga}} \sum_{e_i \in \ee{\ga}}\frac{\li(R_{a_{i},p}-R_{b_{i},p})^2}{(\li+\ri)^2}$ by
\remref{rem independence pdif}, the first equality in the lemma follows if we sum above equality over all
$p \in \vv{\ga}$ and apply \exref{ex disc0}.
Then the second equality in the lemma follows from the fact that
$\sum_{q \in \vv{\ga}}\va(q)l_{pq}^+ = \sum_{e_i \in \ee{\ga}} \big(l_{p \pp}^+ +l_{p \qq}^+ \big).$
\end{proof}
\begin{lemma}\label{lem disc9x}
Let $\lm$ be the discrete Laplacian matrix of a
graph $\ga$. Let $\pp$ and $\qq$ be end points of $e_i \in
\ee{\ga}$. Then
$$\sum_{q, \, s  \in \vv{\ga}} l_{qs} l_{qq}^+  l_{ss}^+
= -\frac{1}{2}\sum_{q, \, s  \in \vv{\ga}} l_{qs} \big( l_{qq}^+ -
l_{ss}^+ \big)^2=\sum_{e_i  \in \ee{\ga}} \frac{1}{\li}\big( l_{\pp
\pp}^+ - l_{\qq \qq}^+ \big)^2 \geq 0.$$
\end{lemma}
\begin{proof}
By \eqnref{eqnlem disc6},
$\sum_{q, \, s  \in \vv{\ga}} l_{qs} \big( l_{qq}^+ - l_{ss}^+
\big)^2 = -2\sum_{q, \, s  \in \vv{\ga}} l_{qs}l_{qq}^+ l_{ss}^+$.
This gives the first equality in the lemma. Then the second equality is obtained by
using the definition of $\lm$.
\end{proof}
In \thmref{thm disc1} below, an important summation term contributing to the tau constant, as can be seen in \propref{proptau}, is expressed in terms of the entries of $\lm$ and $\plm$. This theorem combines various technical lemmas shown above, and it will be used in the proof of \thmref{thm disc2}.
\begin{theorem}\label{thm disc1}
Let $\lm$ be the discrete Laplacian matrix of size $v \times v$ for a metrized
graph $\ga$. Let $\pp$ and $\qq$ be end points of edge $e_i \in \ee{\ga}$, and let $\ri$, $R_{a_i, p}$, $R_{b_i, p}$ and $\li$ be as defined before.
$$\sum_{e_i \in \ee{\ga}} \frac{\li \big( R_{b_i, p}-R_{a_i, p} \big)^2}{(\li + \ri)^2} =
\frac{4}{v}tr(\plm) -\frac{1}{2}\sum_{q, \, s  \in \vv{\ga}} l_{qs} \big( l_{qq}^+ -
l_{ss}^+ \big)^2.$$
\end{theorem}
\begin{proof}
Note that the following equality follows from \exref{ex disc0} for any $p \in \vv{\ga}$,
\begin{equation}\label{eqnlem disc9ab}
\begin{split}
& \sum_{q, \, s \in \vv{\ga}} l_{qs} \big( l_{qp}^+ \big)^2 =
\sum_{q \in \vv{\ga}} \big( l_{qp}^+ \big)^2 \sum_{ s \in \vv{\ga}}
l_{qs} = 0.
\end{split}
\end{equation}
By \corref{cor disc1}, for each $p \in \vv{\ga}$
\begin{equation}\label{eqnlem disc9abc}
\begin{split}
\sum_{q, \, s \in \vv{\ga}} l_{qs} l_{qq}^+  l_{sp}^+  = l_{pp}^+ -\frac{1}{v}tr(\plm).
\end{split}
\end{equation}
Similarly, by \corref{cor disc1} and \remref{rem doubcent1}, for any $p \in \vv{\ga}$ we have
\begin{equation}\label{eqnlem disc9abcd}
\begin{split}
\sum_{q, \, s \in \vv{\ga}} l_{qs} l_{qp}^+  l_{sp}^+  = l_{pp}^+.
\end{split}
\end{equation}
Then for each $p \in \vv{\ga}$,
\begin{equation*}
\begin{split}
& \sum_{e_i \in \ee{\ga}} \frac{\li \big( R_{b_i, p}-R_{a_i, p}
\big)^2}{(\li + \ri)^2}= \sum_{e_i \in
\ee{\ga}}\frac{1}{\li} \big( r(\pp, p)-r(\qq, p) \big)^2, \quad \text{by \eqnref{eqn2termeq1}} \\
& = -\sum_{e_i \in \ee{\ga}} l_{\pp \qq} \big( -2l_{p \pp}^+ +
l_{\pp \pp}^+ + 2l_{p \qq}^+ - l_{\qq \qq}^+ \big)^2, \quad \text{by
\lemref{lem disc2}}
\\ & = -\frac{1}{2}\sum_{q, \, s  \in \vv{\ga}} l_{qs} \big( -2l_{p
q}^+ + l_{qq}^+ + 2l_{p s}^+ - l_{ss}^+ \big)^2\\
& = -\frac{1}{2}\sum_{q, \, s  \in \vv{\ga}} l_{qs} \big( l_{qq}^+ -
l_{ss}^+ \big)^2 +2\sum_{q, \, s  \in \vv{\ga}} l_{qs} \big(
l_{qq}^+ - l_{ss}^+ \big)\big( l_{p q}^+ - l_{p s}^+ \big)
-2\sum_{q, \, s \in \vv{\ga}} l_{qs}\big( l_{p q}^+ - l_{p s}^+
\big)^2
\\
& = -\frac{1}{2}\sum_{q, \, s  \in \vv{\ga}} l_{qs} \big( l_{qq}^+ -
l_{ss}^+ \big)^2 -4\sum_{q, \, s \in \vv{\ga}} l_{qs} l_{qq}^+
l_{sp}^+ -2\sum_{q, \, s \in \vv{\ga}} l_{qs}\big( l_{p q}^+ - l_{p
s}^+ \big)^2, \, \, \ \text{by \lemref{lem disc9}.}
\\
& = -\frac{1}{2}\sum_{q, \, s  \in \vv{\ga}} l_{qs} \big( l_{qq}^+ -
l_{ss}^+ \big)^2 -4\sum_{q, \, s \in \vv{\ga}} l_{qs} l_{qq}^+
l_{sp}^+ +4\sum_{q, \, s \in \vv{\ga}} l_{qs}l_{p q}^+ l_{p s}^+,
\quad \text{by \eqnref{eqnlem disc9ab}.}\\
& = -\frac{1}{2}\sum_{q, \, s  \in \vv{\ga}} l_{qs} \big( l_{qq}^+ -
l_{ss}^+ \big)^2 -4(l_{pp}^+ -\frac{1}{v}tr(\plm))+4(l_{pp}^+),
\quad \text{by Equations (\ref{eqnlem disc9abc}) and (\ref{eqnlem disc9abcd}).}
\end{split}
\end{equation*}
This gives the result.
\end{proof}
In the following lemma, another important summation term contributing to the tau constant (see \propref{proptau}) is expressed in terms of $\lm$ and $\plm$:
\begin{lemma}\label{lem discr1}
Let $\lm$ be the discrete Laplacian matrix of a
metrized graph $\ga$. Suppose $\pp$ and $\qq$ are end points of edge $e_i$. Then
$$\sum_{e_i \in \ee{\ga}} \frac{\li^3}{(\li + \ri)^2} =
\sum_{e_i \in \ee{\ga}}\frac{1}{\li}\big(\li-l_{\pp \pp}^+ + 2l_{\pp
\qq}^+ -l_{\qq \qq}^+ \big)^2.
$$
\end{lemma}
\begin{proof}
Since $\frac{\li \ri}{\li+\ri}=r(\pp,\qq)$ for each $e_i \in \ee{\ga}$, we have
\begin{equation*}
\begin{split}
\sum_{e_i \in \ee{\ga}} \frac{\li^3}{(\li + \ri)^2} = \sum_{e_i \in
\ee{\ga}}\frac{1}{\li} \Big(\li-\frac{\li \ri}{\li +
\ri}\Big)^2=\sum_{e_i \in \ee{\ga}} \frac{1}{\li}
\Big(\li-r(\pp,\qq)\Big)^2.
\end{split}
\end{equation*}
Then the result follows from \lemref{lem disc2}.
\end{proof}
Our main result is the following formula for $\tg$:
\begin{theorem}\label{thm disc2}
Let $\lm$ be the discrete Laplacian matrix of size $v \times v$ for a metrized graph
$\ga$, and let $\plm$ be its pseudo inverse. Suppose $\pp$ and $\qq$ are end points of $e_i \in \ee{\ga}$. Then we have
\begin{equation*}
\begin{split}
\tg & =-\frac{1}{12}\sum_{e_i \in \ee{\ga}} l_{\pp \qq}\big(\frac{1}{ l_{\pp \qq}}+l_{\pp \pp}^+-2l_{\pp
\qq}^+ +l_{\qq \qq}^+ \big)^2 +\frac{1}{4}\sum_{q, \, s  \in
\vv{\ga}} l_{qs} l_{qq}^+  l_{ss}^+ +\frac{1}{v}tr(\plm),
\\ \tg & = -\frac{1}{12}\sum_{e_i \in \ee{\ga}} l_{\pp \qq}\big(\frac{1}{ l_{\pp \qq}}+l_{\pp \pp}^+-2l_{\pp
\qq}^+ +l_{\qq \qq}^+ \big)^2 -\frac{1}{4}\sum_{e_i \in \ee{\ga}} l_{p_i q_i} \big(l_{\pp \pp}^+ -  l_{\qq \qq}^+\big)^2 +\frac{1}{v}tr(\plm).
\end{split}
\end{equation*}
\end{theorem}
\begin{proof}
By \propref{proptau}, for any $p \in \vv{\ga}$
\begin{equation*}
\begin{split}
\tg=\frac{1}{12}\sum_{e_i \in \ee{\ga}}\frac{\li^3}{(\li +
\ri)^2}+\frac{1}{4}\sum_{e_i \in \ee{\ga}} \frac{\li \big( R_{b_i,
p}-R_{a_i, p} \big)^2}{(\li + \ri)^2}.
\end{split}
\end{equation*}
Thus the first equality in the theorem follows from \lemref{lem discr1}, \thmref{thm disc1}
and \lemref{lem disc9x}. Then the second equality follows from \lemref{lem disc9x}.
\end{proof}
\begin{corollary}\label{cor discr lower bound}
Let $\lm$ be the discrete Laplacian matrix of size $v \times v$ for a graph
$\ga$. Then we have
$\tg \geq \frac{1}{v}tr(\plm).$
\end{corollary}

Using \thmref{thm disc2 copy} and \eqnref{eqn disc5}, we can compute $\tg$ and $\mu_{can}$ by a computer program whose computational complexity and memory consumption when $\tg$ is computed is at the level of a matrix inversion.

Next, we will express $\mu_{can}$ in terms of the discrete Laplacian matrix and its pseudo inverse.
\begin{proposition}\label{prop mucan}
For a given metrized graph $\ga$, let $\lm$ be its discrete Laplacian, and let $\plm$ be the corresponding pseudo inverse.
Suppose $\pp$ and $\qq$ denotes the end points of $e_i \in \ee{\ga}$. Then we have
$$\mu_\can(x) = \sum_{p \in \vv{\ga}} (1 - \frac{1}{2}\vb(p))
\, \delta_p(x) - \sum_{e_i \in \ee{\ga}} \big(l_{\pp \qq}+l_{\pp \qq}^2 (l_{\pp \pp}^+ -2 l_{\pp \qq}^+ +l_{\qq \qq}^+)\big)dx.$$
\end{proposition}
\begin{proof}
The result follows from \thmref{thmCanonicalMeasureFormula}, \lemref{lem disc2}, and the fact that $r(\pp,\qq)=\frac{\li \ri}{\li+\ri}$
for each $e_i \in \ee{\ga}$.
\end{proof}
In the rest of this section, we will compute
the tau constant and the canonical measure for some metrized graphs.
\begin{example}
Let $\ga$ be a complete graph on $5$ vertices with each edge length is equal to $\frac{1}{10}$, so that $\elg=1$.
Then $\ga$ has the following discrete Laplacian matrix and pseudo inverse:
\[
\lm=\left[
\begin{array}{ccccc}
 40 & -10 & -10 & -10 & -10 \\
 -10 & 40 & -10 & -10 & -10 \\
 -10 & -10 & 40 & -10 & -10 \\
 -10 & -10 & -10 & 40 & -10 \\
 -10 & -10 & -10 & -10 & 40
\end{array}
\right]
\text{\, and \, \,}
\plm=\left[
\begin{array}{ccccc}
 \frac{2}{125} & -\frac{1}{250} & -\frac{1}{250} & -\frac{1}{250} & -\frac{1}{250} \medskip \\
 -\frac{1}{250} & \frac{2}{125} & -\frac{1}{250} & -\frac{1}{250} & -\frac{1}{250} \medskip \\
 -\frac{1}{250} & -\frac{1}{250} & \frac{2}{125} & -\frac{1}{250} & -\frac{1}{250} \medskip \\
 -\frac{1}{250} & -\frac{1}{250} & -\frac{1}{250} & \frac{2}{125} & -\frac{1}{250} \medskip \\
 -\frac{1}{250} & -\frac{1}{250} & -\frac{1}{250} & -\frac{1}{250} & \frac{2}{125}
\end{array}
\right].
\]
Thus, we obtain $\tg=\frac{23}{500}$ by applying \thmref{thm disc2}.
Moreover, \propref{prop mucan} can be used to compute the canonical measure of $\ga$.
Namely, $$\mu_{can}(x)=-\sum_{p \in \vv{\ga}}\delta_{p}(x)+6\sum_{e_i \in \ee{\ga}}dx.$$
%
\end{example}
\begin{figure}
\centering
\includegraphics[scale=0.8]{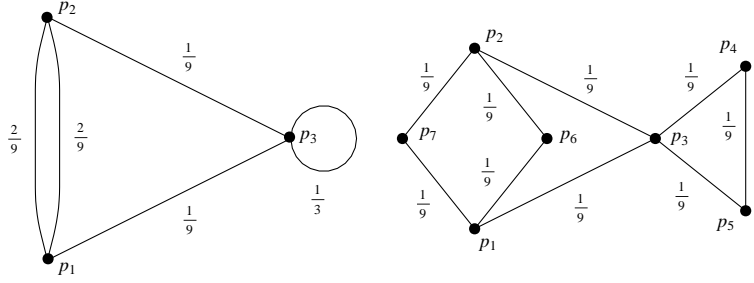} \caption{$\ga$ with $\vv{\ga}=\{ 1, \, 2, \, 3\}$ and with an optimal vertex set $\{ 1, \, 2, \, 3, \, 4, \, 5, \, 6, \, 7\}$.} \label{fig g1andg2}
\end{figure}

\begin{example}
Let $\ga$ be a metrized graph illustrated as the first graph in \figref{fig g1andg2}, where the edge lengths are also shown. Note that $\elg=1$.
Since $\ga$ with this set of vertices $\vv{\ga}=\{ 1, \, 2, \, 3\}$ has a self loop and two multiple edges, we need to work with an optimal
vertex set to have the associated discrete Laplacian matrix. This is done by considering additional $2$ points on the self loop as new vertices, and taking
$1$ more points on each multiple edges as new vertices. The new metrized graph is illustrated by the second graph in \figref{fig g1andg2}. As we know by \remref{remvalence} that the new length distribution for the self loop and the multiple edges
will not change $\tg$. Now, $\ga$ has the following discrete Laplacian matrix and the pseudo inverse:
\[
\lm=\left[
\begin{array}{ccccccc}
 27 & 0 & -9 & 0 & 0 & -9 & -9 \\
 0 & 27 & -9 & 0 & 0 & -9 & -9 \\
 -9 & -9 & 36 & -9 & -9 & 0 & 0 \\
 0 & 0 & -9 & 18 & -9 & 0 & 0 \\
 0 & 0 & -9 & -9 & 18 & 0 & 0 \\
 -9 & -9 & 0 & 0 & 0 & 18 & 0 \\
 -9 & -9 & 0 & 0 & 0 & 0 & 18
\end{array}
\right]
\]
and
\[
\plm=\left[
\begin{array}{ccccccc}
 \frac{47}{1323} & -\frac{2}{1323} & -\frac{1}{147} & -\frac{10}{441} & -\frac{10}{441} & \frac{4}{441} & \frac{4}{441} \medskip \\
 -\frac{2}{1323} & \frac{47}{1323} & -\frac{1}{147} & -\frac{10}{441} & -\frac{10}{441} & \frac{4}{441} & \frac{4}{441} \medskip \\
 -\frac{1}{147} & -\frac{1}{147} & \frac{11}{441} & \frac{4}{441} & \frac{4}{441} & -\frac{13}{882} & -\frac{13}{882} \medskip \\
 -\frac{10}{441} & -\frac{10}{441} & \frac{4}{441} & \frac{89}{1323} & \frac{40}{1323} & -\frac{3}{98} & -\frac{3}{98} \medskip \\
 -\frac{10}{441} & -\frac{10}{441} & \frac{4}{441} & \frac{40}{1323} & \frac{89}{1323} & -\frac{3}{98} & -\frac{3}{98} \medskip \\
 \frac{4}{441} & \frac{4}{441} & -\frac{13}{882} & -\frac{3}{98} & -\frac{3}{98} & \frac{25}{441} & \frac{1}{882} \medskip \\
 \frac{4}{441} & \frac{4}{441} & -\frac{13}{882} & -\frac{3}{98} & -\frac{3}{98} & \frac{1}{882} & \frac{25}{441}
\end{array}
\right].
\]
Finally, applying \thmref{thm disc2} gives $\tg=\frac{23}{324}$.
By \propref{prop mucan}, we have the following canonical measure for $\ga$:
$$\mu_{can}(x)=-\frac{1}{2}\delta_{p_1}(x)-\frac{1}{2}\delta_{p_2}(x)-\delta_{p_3}(x)+3\sum_{e_i \in \ee{\ga}}dx.$$
%
\end{example}


\end{document}